\documentclass[12pt,reqno]{amsart}
\pdfoutput=1
\usepackage[margin=1in]{geometry}
\usepackage{amssymb}
\usepackage{thmtools}
\usepackage{enumerate}
\usepackage{url}
\usepackage{shuffle}
\usepackage[T1]{fontenc}
\usepackage{hyperref}
\usepackage[capitalize,noabbrev,nameinlink]{cleveref}

\usepackage{tikz-cd}

\DeclareMathOperator{\Des}{Des}
\DeclareMathOperator{\Val}{Val}

\DeclareMathOperator{\dis}{dis}
\DeclareMathOperator{\inv}{inv}
\DeclareMathOperator{\stan}{stan}
\DeclareMathOperator{\triv}{triv}
\DeclareMathOperator{\Pk}{Pk}
\DeclareMathOperator{\sPk}{sPk}

\DeclareMathOperator{\maj}{maj}
\DeclareMathOperator{\lift}{\uparrow}
\DeclareMathOperator{\Comp}{Comp}

\newcommand{\sh}{\shuffle}
\newcommand{\Q}{\mathbb{Q}}
\newcommand{\QSym}{\textnormal{QSym}}
\newcommand{\SSym}{\textnormal{SSym}}

\newcommand{\st}{\textnormal{st}}

\newcommand{\A}{\mathcal{A}}
\newcommand{\C}{\mathcal{C}}
\renewcommand{\H}{\mathcal{H}}

\renewcommand{\a}{\alpha}

\newtheorem{theorem}{Theorem}[section]
\newtheorem{proposition}[theorem]{Proposition}
\newtheorem{corollary}[theorem]{Corollary}
\newtheorem{lemma}[theorem]{Lemma}

\theoremstyle{definition}

\newtheorem{example}[theorem]{Example}
\newtheorem{defn}[theorem]{Definition}
\newtheorem{conjecture}[theorem]{Conjecture}

\theoremstyle{remark}
\newtheorem{remark}[theorem]{Remark}

\numberwithin{equation}{section}

\usepackage[style=alphabetic,doi=false,isbn=false,eprint=false,maxnames=50]{biblatex}
\title{Substring compatibility of permutation statistics}
\date{\today}
\author{Michael Tang}
\address{Department of Mathematics, University of Washington, Seattle, WA 98195}
\email{mst0@uw.edu}

\begin{document}

\begin{abstract}
    A permutation statistic is \emph{substring-compatible} if its value on a permutation determines its value on every substring of that permutation. We construct the \emph{substring coalgebra} of such a statistic, an analog of the shuffle algebra of a shuffle-compatible statistic introduced by Gessel and Zhuang. Furthermore, we show that for substring-compatible statistics that also satisfy a weak form of shuffle compatibility, the shuffle algebra and substring coalgebra can be combined to yield a Hopf algebra. Finally, we conjecture that the only nontrivial permutation statistics that are both shuffle-compatible and substring-compatible are the descent set, the peak set, and the valley set, and we describe our progress towards proving this conjecture.
\end{abstract}

\maketitle

\section{Introduction} \label{sec:intro}

Define a \emph{permutation} to be a sequence $\pi = \pi_1 \dotsm \pi_n$ of distinct positive integers $\pi_i$, called the letters of $\pi$; we use $|\pi| = n$ to denote the length of $\pi$. Then a \emph{permutation statistic}, or just \emph{statistic} for short, is a function $\st$ defined on permutations such that $\st(\pi) = \st(\sigma)$ whenever $\pi$ and $\sigma$ have the same relative order (that is, $\pi_i < \pi_j$ if and only if $\sigma_i < \sigma_j$).

If $\pi = \pi_1 \dotsm \pi_n$ and $\sigma = \sigma_1 \dotsm \sigma_m$ are permutations with no letters in common, a \emph{shuffle} of $\pi$ and $\sigma$ is a permutation of length $n+m$ that contains both $\pi$ and $\sigma$ as subsequences; the \emph{shuffle set} $\pi \sh \sigma$ is the set of all $\binom{n+m}{n}$ shuffles of $\pi$ and $\sigma$. We say that $\st$ is \emph{shuffle-compatible} if the multiset $\{\st(\tau) : \tau \in \pi \sh \sigma\}$ only depends on $\st(\pi)$ and $\st(\sigma)$, as well as $|\pi|$ and $|\sigma|$.

It is a remarkable fact, known since the early work of Stanley, that many important permutation statistics are shuffle-compatible, such as the \emph{descent set}, the \emph{peak set}, the \emph{valley set}, and the \emph{major index}. Shuffle compatibility was formally defined later by Gessel and Zhuang \cite{gessel-zhuang}, who also introduced an algebraic framework for studying it: they constructed the \emph{shuffle algebra} $\A_\st$ of a shuffle-compatible statistic $\st$, and showed that many proofs of shuffle compatibility can be expressed in terms of shuffle algebras. For example, Stanley's work on $P$-partitions \cite{stanleythesis} shows that the shuffle algebra of the descent set is isomorphic to $\QSym$, the ring of quasisymmetric functions. Additionally, Stembridge's work \cite{stembridge} shows that the shuffle algebra of the peak set is isomorphic to a subalgebra $\Pi$ of $\QSym$ which he called the ``algebra of peaks.''

Our part of the story begins with the observation that $\QSym$ and $\Pi$, unlike the shuffle algebras of the descent set and peak set, possess additional structure that make them into coalgebras and thus Hopf algebras. To describe this structure in terms of permutation statistics, we introduce a property called substring compatibility. A permutation statistic $\st$ is said to be \emph{substring-compatible} if its value on a permutation $\pi = \pi_1 \dotsm \pi_n$ of given length determines the value of $\st(\pi_i \pi_{i+1} \dotsm \pi_j)$ for all $1 \le i \le j \le n$. As we will see, the descent set, peak set, and valley set are all substring-compatible. Furthermore, substring compatibility can be described succinctly using the refinement partial order on permutation statistics together with an operation on statistics called \emph{lifting}.

Given a substring-compatible statistic $\st$, we construct its \emph{substring coalgebra} $\C_\st$, which serves as an analog and counterpart to the shuffle algebra. Then, for statistics $\st$ that are substring-compatible and that also satisfy a certain weaker form of shuffle compatibility (under which the shuffle algebra can still be defined), we prove that $\A_\st$ and $\C_\st$ can be combined to form a Hopf algebra $\H_\st$. We show that $\H_\st$ is naturally a quotient of $\SSym$, the Malvenuto-Reutenauer Hopf algebra of permutations \cite{malvenuto-reutenauer}, and we verify that $\H_{\Des}$ and $\H_{\Pk}$ are isomorphic as Hopf algebras to $\QSym$ and $\Pi$, respectively.

Finally, we consider the problem of determining all permutation statistics that are both shuffle-compatible and substring-compatible. We conjecture (\cref{conj:bicompatible}) that, up to equivalence,\footnote{As defined in \cref{sec:shuffle-compatibility}, permutation statistics $\st$ and $\st'$ are said to be \emph{equivalent} if for permutations $\pi$ and $\sigma$ of the same length, $\st(\pi) = \st(\sigma)$ if and only if $\st'(\pi) = \st'(\sigma)$.} there are exactly four such statistics: the descent set, the peak set, the valley set, and the trivial statistic (which takes the same value for every permutation). We make progress towards proving this conjecture and give evidence, obtained by the use of a SAT solver, that the conjecture is true. In particular, we prove \cref{thm:bicompatible-progress}, which says that any other shuffle-compatible and substring-compatible statistic must be trivial on permutations of length at most $6$: that is, if $\st$ is such a statistic, then $\st(\pi) = \st(\sigma)$ whenever $|\pi| = |\sigma| \le 6$. We are also able to impose additional significant restrictions on the existence of such a statistic. For instance, we show (\cref{thm:number-theoretic}) that if $\st$ is shuffle-compatible, substring-compatible, and trivial on all permutations of length at most $n-1$, where $n$ is not a power of a prime, then $\st$ must also be trivial on permutations of length $n$.

This paper is organized as follows. In \cref{sec:shuffle-compatibility}, we give some background about permutations and permutation statistics, and define the weak form of shuffle compatibility. In \cref{sec:substring-compatibility}, we formally define substring compatibility, introduce the substring coalgebra, and show how it can be combined with the shuffle algebra. In \cref{sec:characterizing-bicompatible}, we describe our progress in characterizing all statistics which are both shuffle-compatible and substring-compatible.

\section{Shuffle compatibility} \label{sec:shuffle-compatibility}

Throughout this paper, we will assume familiarity with (associative and unital) algebras, coalgebras, bialgebras, and Hopf algebras; one possible reference is \cite[\S 1]{grinberg}. We always take the rational numbers $\Q$ as the base field. While quasisymmetric functions serve as an important example for our work, we will not need to work with them in great detail.

\subsection{Permutations and permutation statistics}

We begin with some basic definitions.

\begin{defn}
    Let $\pi = \pi_1 \dotsm \pi_n$ and $\sigma = \sigma_1 \dotsm \sigma_m$ be permutations.

    \begin{enumerate}[(i)]
        \item The \emph{length} of $\pi$, denoted $|\pi|$, is the number of letters of $\pi$ (in this case, $n$).
        \item We say $\pi$ is a \emph{standard} permutation if its letters comprise the set $[n] = \{1, \dots, n\}$. Thus, the standard permutations of length $n$ form the symmetric group $S_n$.
        \item The \emph{standardization} of $\pi$ is the unique standard permutation $\stan(\pi) \in S_n$ with the same relative order as $\pi$. In other words, it is obtained by replacing the $i^{\text{th}}$ smallest letter of $\pi$ with $i$ for all $1 \le i \le n$.
        \item We say that $\pi$ and $\sigma$ are \emph{disjoint} if they have no letters in common (that is, $\pi_i \neq \sigma_j$ for all $1 \le i \le n$ and $1 \le j \le m$).
        \item If $\pi$ and $\sigma$ are disjoint, then their \emph{concatenation}\footnote{The notation $\pi \sigma$ should not be confused with the usual product of permutations in the symmetric group.} $\pi\sigma$ is the permutation $\pi_1 \dotsm \pi_n \sigma_1 \dotsm \sigma_m$ of length $n+m$. 
        \item If $\pi$ and $\sigma$ are disjoint,\footnote{For brevity, we will not always explicitly state the disjointness condition when shuffling permutations; one should assume that $\pi$ and $\sigma$ are disjoint whenever we speak of shuffles of $\pi$ and $\sigma$ or write $\pi \sh \sigma$.} then a \emph{shuffle} of $\pi$ and $\sigma$ is a permutation of length $n+m$ that contains both $\pi$ and $\sigma$ as subsequences. The \emph{shuffle set} $\pi \sh \sigma$ is the set of all $\binom{n+m}{n}$ shuffles of $\pi$ and $\sigma$.
    \end{enumerate}
\end{defn}

\begin{example}
    The shuffle set of $13$ and $524$ is \[13 \sh 524 = \{13524, 15324, 15234, 15243, 51324, 51234, 51243, 52134, 52143, 52413\}.\]
\end{example}

Next, we define some common permutation statistics.

\begin{defn}
    Let $\pi = \pi_1 \dotsm \pi_n$ be a permutation.

    \begin{enumerate}[(i)]
        \item A \emph{descent} of a permutation $\pi$ is an index $i$ for which $\pi_i > \pi_{i+1}$. The \emph{descent set} $\Des$ is the set of all descents, $\Des(\pi) = \{1 \le i \le n-1 : \pi_i > \pi_{i+1}\}$.
        \item The \emph{major index} $\maj$ is the sum of all descents: $\maj(\pi) = \sum_{i \in \Des(\pi)} i$.
        \item A \emph{peak} (resp.\ \emph{valley}) of $\pi$ is an index $2 \le i \le n-1$ for which $\pi_{i-1} < \pi_i > \pi_{i+1}$ (resp.\ $\pi_{i-1} > \pi_i < \pi_{i+1}$). The \emph{peak set} $\Pk$ is the set of all peaks, and the \emph{valley set} $\Val$ is the set of all valleys.
        \item An \emph{inversion} of $\pi$ is a pair of indices $i < j$ for which $\pi_i > \pi_j.$ Then the \emph{inversion number} $\inv$ is the number of inversions, $\inv(\pi) = \left| \{1 \le i < j \le n : \pi_i > \pi_j \} \right|$.
        \item The \emph{discrete statistic} $\dis$ sends $\pi$ to its standardization. Thus, $\dis(\pi) = \dis(\sigma)$ if and only if $\pi$ and $\sigma$ have the same relative order.
        \item The \emph{trivial statistic} $\triv$ is defined by $\triv(\pi) = 0$ for all $\pi$.
    \end{enumerate}
\end{defn}

\begin{example}
    The permutation $\pi = 29546$ has $\Des(\pi) = \{2,3\}$, $\maj(\pi) = 5$, $\Pk(\pi) = \{2\}$, $\Val(\pi) = \{4\}$, and $\inv(\pi) = 4$.
\end{example}

Each permutation statistic induces an equivalence relation defined as follows.

\begin{defn}[{\cite[\S 3.1]{gessel-zhuang}}]
    Given a permutation statistic $\st$, we say that permutations $\pi$ and $\sigma$ are \emph{$\st$-equivalent}, written $\pi \sim_{\st} \sigma$, if $|\pi| = |\sigma|$ and $\st(\pi) = \st(\sigma)$. We write $[\pi]_\st$ for the equivalence class of a permutation $\pi$ under $\sim_\st$.
\end{defn}

Note that permutations of different lengths are never considered to be $\st$-equivalent. For example, even though the permutations $132$ and $1324$ have the same descent set $\{2\}$, it is not true that $132 \sim_{\Des} 1324$.

Given two permutation statistics $\st$ and $\st'$, we say that they are \emph{equivalent}, written $\st \equiv \st'$, if the equivalence relations $\sim_\st$ and $\sim_{\st'}$ are the same; that is, $\pi \sim_\st \sigma$ if and only if $\pi \sim_{\st'} \sigma$. In the rest of this paper, we will usually consider permutation statistics only up to equivalence and not concern ourselves with the exact values of the statistics.

\subsection{Refinement of permutation statistics}

Permutation statistics are organized via a partial order called refinement \cite[\S 3.1]{gessel-zhuang}.

\begin{defn}
    Let $\st$ and $\st'$ be permutation statistics. We say that $\st$ \emph{refines} $\st'$, written $\st \preceq \st'$, if $\st(\pi)$ and $|\pi|$ determine $\st'(\pi)$ for all permutations $\pi$. In terms of the equivalence relations induced by these statistics, $\st$ refines $\st'$ if $\pi \sim_{\st} \sigma$ implies $\pi \sim_{\st'} \sigma$ for all $\pi$ and $\sigma$.
\end{defn}

We note that $\st \equiv \st'$ if and only if both $\st \preceq \st'$ and $\st' \preceq \st$.

\begin{example}
    For every permutation statistic $\st$, we have $\dis \preceq \st \preceq \triv$.
    
    A permutation statistic $\st$ such that $\Des \preceq \st$ is called a \emph{descent statistic} \cite[\S 2.1]{gessel-zhuang}. It is easy to see that $\maj$, $\Pk$, $\Val$, and $\triv$ are all descent statistics. On the other hand, $\inv$ is not a descent statistic: for example, the permutations $132$ and $231$ have the same descent set $\{2\}$ but different numbers of inversions ($1$ and $2$, respectively).
\end{example}

\subsection{Shuffle compatibility and the shuffle algebra}

First, we recall the definition of shuffle compatibility:

\begin{defn} \label{def:shuffle-compatible}
    A permutation statistic $\st$ is \emph{shuffle-compatible} if for disjoint permutations $\pi$ and $\sigma$, the multiset $\{\st(\tau) : \tau \in \pi \sh \sigma\}$ is determined by $\st(\pi)$, $\st(\sigma)$, $|\pi|$, and $|\sigma|$.
\end{defn}

\begin{example} \label{ex:shuffle-compatible-statistics}
    As stated in \cref{sec:intro}, the descent set, peak set, valley set, and major index are all shuffle-compatible.\footnote{See \cite[Appendix A]{gessel-zhuang} for a much more complete list of shuffle-compatible permutation statistics.} However, the inversion number is not shuffle-compatible: for example, the shuffle sets \[
        12 \sh 3 = \{123, 132, 312\} \qquad \text{and} \qquad 13 \sh 2 = \{123, 132, 213\}
    \] do not have the same distribution under $\inv$ (the corresponding multisets are $\{0,1,2\}$ and $\{0,1,1\}$, respectively), even though $12 \sim_{\inv} 13$ and $3 \sim_{\inv} 2$.
\end{example}

\begin{example}
    The trivial statistic is obviously shuffle-compatible, but the discrete statistic is not (the shuffles $12 \sh 3$ and $13 \sh 2$ again serve as a counterexample).
\end{example}

We can give an alternative description of shuffle compatibility in terms of $\sim_{\st}$. It will be helpful to commit a slight abuse of notation: given sets of permutations $S$ and $S'$, we write $S \sim_\st S'$ if the multisets $\{\st(\pi) : \pi \in S\}$ and $\{\st(\pi) : \pi \in S'\}$ are equal. Equivalently, $S \sim_\st S'$ if there is a bijection $\iota \colon S \to S'$ such that $\pi \sim_\st \iota(\pi)$ for all $\pi \in S$. (Note that if $\st$ refines $\st'$, then $S \sim_\st S'$ still implies $S \sim_{\st'} S'$.)

\begin{proposition}
    A permutation statistic $\st$ is shuffle-compatible if and only if whenever $\pi$, $\pi'$, $\sigma$, and  $\sigma'$ are permutations with $\pi \sim_\st \pi'$ and $\sigma \sim_\st \sigma'$, we have $\pi \sh \sigma \sim_\st \pi' \sh \sigma'$.
\end{proposition}

\begin{proof}
    This follows directly from \cref{def:shuffle-compatible} (assuming disjointness as necessary).
\end{proof}

We now define the shuffle algebra of a shuffle-compatible permutation statistic.

\begin{defn}[{\cite[\S 3.1]{gessel-zhuang}}] \label{def:shuffle-algebra}
    Let $\st$ be a shuffle-compatible permutation statistic. The \emph{shuffle algebra} of $\st$ is the graded $\Q$-algebra $\A_\st$ with basis $[\pi]_\st$ indexed by equivalence classes of $\sim_\st$, graded by length, and multiplication defined by
    \begin{align}
        [\pi]_\st \cdot [\sigma]_\st = \sum_{\tau \in \pi \sh \sigma} [\tau]_\st \label{eq:shuffle-algebra-product}
    \end{align}
    and extended linearly, where $\pi$ and $\sigma$ are disjoint permutations.
\end{defn}

Note that shuffle compatibility implies that this multiplication is well-defined, i.e., that the formula for $[\pi]_\st \cdot [\sigma]_\st$ does not depend on the choice of representatives $\pi$ and $\sigma$.

If one shuffle-compatible statistic refines another, it is easy to check that their shuffle algebras are related in the following way (cf.\ \cite[Theorem 3.3]{gessel-zhuang}):

\begin{proposition} \label{prop:shuffle-algebra-refinement}
    If $\st$ and $\st'$ are shuffle-compatible statistics and $\st \preceq \st'$, then the map $\phi : \A_\st \to \A_{\st'}$ defined by $\phi([\pi]_\st) = [\pi]_{\st'}$ is a surjective algebra homomorphism. Thus, $\A_{\st'}$ is isomorphic to the quotient $\A_{\st} \, / \ker(\phi)$.
\end{proposition}

Arguably the most important example of a shuffle algebra is $\A_{\Des}$, the shuffle algebra of the descent set statistic. Stanley's theory of $P$-partitions \cite{stanleythesis} implies that $\A_{\Des}$ is isomorphic as an algebra to $\QSym$, the \emph{ring of quasisymmetric functions}. In \cref{sec:des-pk-hopf}, we will examine the descent set, as well as the closely related peak set, in more detail.

\subsection{Weak shuffle compatibility}

To conclude this section, we describe a weaker variant of shuffle compatibility for later use. This property is mentioned briefly in \cite[\S 1]{gessel-zhuang}, and some statistics with this property are studied by Vong \cite{vong}.

For permutations $\pi$ and $\sigma$, write $\pi < \sigma$ if every letter of $\pi$ is strictly less than every letter of $\sigma$. Then, we obtain weak shuffle compatibility by requiring $\pi < \sigma$ in all shuffles $\pi \sh \sigma$.

\begin{defn}
    A permutation statistic $\st$ is \emph{weakly shuffle-compatible} if whenever $\pi$ and $\sigma$ are permutations with $\pi < \sigma$, the multiset $\{\st(\tau) : \tau \in \pi \sh \sigma\}$ only depends on $\st(\pi)$, $\st(\sigma)$, $|\pi|$, and $|\sigma|$. (Equivalently, if $\pi \sim_\st \pi'$, $\sigma \sim_\st \sigma'$, $\pi < \sigma$, and $\pi' < \sigma'$, then $\pi \sh \sigma \sim \pi' \sh \sigma'$.)
\end{defn}

\begin{example}
    As we observed earlier, the discrete statistic $\dis$ is not shuffle-compatible; however, it is weakly shuffle-compatible. Indeed, suppose $\pi$ and $\sigma$ are permutations with $\pi < \sigma$. Then, given $\dis(\pi)$ and $\dis(\sigma)$, we know the relative order of all letters in $\pi$ and $\sigma$, and thus we know the set of all shuffles of $\pi$ and $\sigma$ up to standardization.
\end{example}

\begin{example}
    Similarly, the inversion number $\inv$ is not shuffle-compatible, but it is weakly shuffle-compatible. Let $\pi = \pi_1 \dotsm \pi_m$ and $\sigma = \sigma_1 \dotsm \sigma_n$ be permutations with $\pi < \sigma$. Then for each shuffle $\tau$ of $\pi$ and $\sigma$, we have that \[
        \inv(\tau) = \inv(\pi) + \inv(\sigma) + \left| \{ i < j : \tau_i \text{ is a letter of } \sigma \text{ and } \tau_j \text{ is a letter of } \pi \} \right|.
    \] It follows that the multiset $\{\inv(\tau) : \tau \in \pi \sh \sigma\}$ is completely determined by $\inv(\pi)$ and $\inv(\sigma)$, as well as $m$ and $n$.
\end{example}

\begin{example} \label{ex:dual-knuth-weakly-shuffle}
    Given a permutation $\pi$, a \emph{dual Knuth move} consists of swapping the $k^{\text{th}}$ and $(k+1)^{\text{th}}$ smallest letters in $\pi$ for some $k$, given that the $(k-1)^{\text{th}}$ smallest letter or $(k+2)^{\text{th}}$ smallest letter (or both) appear between them in $\pi$. For instance, if $\pi = 241536$, then we can swap $3$ and $4$ in $\pi$ with a dual Knuth move, but we cannot swap $4$ and $5$.
    
    Permutations $\pi$ and $\sigma$ are called \emph{dual Knuth equivalent}, written $\pi \sim_{\text{dK}} \sigma$, if one can be reached from the other by a sequence of dual Knuth moves. Knuth \cite{knuth} proved that standard permutations are dual Knuth equivalent if and only if, under the RSK correspondence $\pi \mapsto (P(\pi), Q(\pi))$, they have the same \emph{recording tableau} $Q(\pi)$. It follows that dual Knuth equivalence is induced by the permutation statistic that sends $\pi$ to $Q(\pi)$.
    
    We claim that dual Knuth equivalence is weakly shuffle-compatible. Let $\pi$, $\sigma$, and $\sigma'$ be permutations with $\pi < \sigma$ and $\pi < \sigma'$, and suppose that $\sigma$ and $\sigma'$ are related by a dual Knuth move swapping the letters $a$ and $b$. Then in any shuffle of $\pi$ and $\sigma$, the dual Knuth move swapping $a$ and $b$ is still valid and yields a corresponding shuffle of $\pi$ and $\sigma'$. Therefore, we get $\pi \sh \sigma \sim_{\text{dK}} \pi \sh \sigma'$. Similarly, given $\pi$, $\pi'$, and $\sigma$ with $\pi < \sigma$ and $\pi' < \sigma$ where $\pi$ and $\pi'$ are related by a dual Knuth move, we have $\pi \sh \sigma \sim_{\text{dK}} \pi' \sh \sigma$. The claim follows by transitivity.
\end{example}

If a permutation statistic $\st$ is weakly shuffle-compatible (but not shuffle-compatible), we can still define its shuffle algebra $\A_\st$ using \cref{def:shuffle-algebra}; the only difference is that we require $\pi < \sigma$ in the formula \eqref{eq:shuffle-algebra-product} for the product. Then, it is clear that these two notions of the shuffle algebra agree for shuffle-compatible permutation statistics. Furthermore, \cref{prop:shuffle-algebra-refinement} still holds when $\st$ or $\st'$ (or both) are only weakly shuffle-compatible.

\begin{example} \label{ex:ssym}
    The \emph{Malvenuto-Reutenauer Hopf algebra} $\SSym$, introduced in \cite{malvenuto-reutenauer}, is a graded $\Q$-algebra with basis $\{\mathcal{F}_\pi\}$ indexed by standard permutations $\pi$, graded by length, and multiplication given for standard $\pi = \pi_1 \dotsm \pi_m$ and $\sigma = \sigma_1 \dotsm \sigma_n$ by \[
        \mathcal{F}_\pi \cdot \mathcal{F}_\sigma = \sum_{\tau \in \pi \sh \sigma[m]} \mathcal{F}_\tau
    \] where $\sigma[m]$ is the length-$n$ permutation with $(\sigma[m])_i = m + \sigma_i$. (Thus, $\sigma[m]$ has the same relative order as $\sigma$, and $\pi < \sigma[m]$.) It is clear from this formula that $\A_{\dis}$ is isomorphic to $\SSym$ under the map $[\pi]_{\dis} \mapsto \mathcal{F}_{\stan(\pi)}$. Since all permutation statistics refine $\dis$, it follows from \cref{prop:shuffle-algebra-refinement} that every shuffle algebra $\mathcal{A}_\st$ is a quotient of $\SSym$.
\end{example}

\section{Substring compatibility} \label{sec:substring-compatibility}

We now come to the main focus of this paper: substring compatibility and how it interacts with shuffle compatibility.

\begin{defn}
    A permutation statistic $\st$ is \emph{substring-compatible} if given a permutation $\pi$, the values of $\st(\pi)$ and $|\pi|$ determine $\st(\pi_i \dotsm \pi_j)$ for all $1 \le i \le j \le |\pi|$. Alternatively, $\st$ is substring-compatible if $\pi \sim_{\st} \sigma$ implies that $\pi_i \dotsm \pi_j \sim_{\st} \sigma_i \dotsm \sigma_j$ for all $1 \le i \le j \le |\pi|$.
\end{defn}

Here are some examples (and non-examples) of substring-compatible statistics.

\begin{example}
    Clearly, the discrete statistic and trivial statistic are substring-compatible.
\end{example}

\begin{example}
    For any permutation $\pi = \pi_1 \dotsm \pi_n$ and any $1 \le i \le j \le n$, we have \[
        \Des(\pi_i \dotsm \pi_j) = \Des(\pi) \cap \{i, \dots, j-1\}.
    \] Therefore, $\Des(\pi)$ determines $\Des(\pi_i \dotsm \pi_j)$, so the descent set $\Des$ is substring-compatible. A similar argument shows that the peak set $\Pk$ and valley set $\Val$ are also substring-compatible.
\end{example}

\begin{example} \label{ex:dual-knuth-substring}
    We claim that dual Knuth equivalence is substring-compatible. To prove this, suppose that permutations $\pi = \pi_1 \dotsm \pi_n$ and $\sigma = \sigma_1 \dotsm \sigma_n$ are related by a dual Knuth move that swaps letters $a$ and $b$. Let $\pi' = \pi_i \dotsm \pi_j$ and $\sigma' = \sigma_i \dotsm \sigma_j$ be corresponding substrings of $\pi$ and $\sigma$, respectively. If both substrings contain $a$ and $b$, then swapping them is still a valid dual Knuth move that relates $\pi'$ and $\sigma'$. Otherwise, $\pi'$ and $\sigma'$ have the same relative order (since $\pi'$ and $\sigma'$ have no letters with value strictly between $a$ and $b$). Thus, we always have $\pi' \sim_{\text{dK}} \sigma'$. The claim then follows by transitivity.
\end{example}

\begin{example}
    The major index $\maj$ and inversion number $\inv$ are not substring-compatible. (Informally, these statistics do not carry enough ``local'' information about a permutation to recover data about all the substrings of the permutation.) For example, $3214$ and $2341$ have both the same major index and the same inversion number, but the corresponding substrings $32$ and $23$ have different major indices and different inversion numbers.
\end{example}

\subsection{Lifts of permutation statistics}

We now define the \emph{lifting} operation on permutation statistics, which allows us to give a useful alternative description of substring compatibility.

It will be helpful to allow for permutation statistics that are defined only for permutations of a certain length $n$ (``statistics defined on length $n$,'' for short), or only for permutations of length at most $n$ (``statistics defined on length $\le\!n$''). We denote such a statistic by $\st_n$ or $\st_{\le n}$, respectively; we also use this notation to mean the restriction of a statistic to permutations of length $n$ or at most $n$. In particular, when we write a statistic $\st$ without a subscript, we require it to be defined for all permutations.

Note that the definitions of refinement $\preceq$ and equivalence $\equiv$ can easily be extended to statistics defined on length $n$ or statistics defined on length $\le\!n$. The notions of (weak) shuffle compatibility and substring compatibility also continue to make sense for statistics defined on length $\le\!n$, assuming that all permutations involved have length at most $n$.

\begin{defn}
    Let $\st_n$ be a permutation statistic defined on length $n$. For nonnegative integers $k$, the $k$-\emph{lift} of $\st_n$, denoted $\lift^k \st_n$, is the statistic defined on length $n+k$ by \[
        \lift^k \st_n(\pi_1 \dots \pi_{n+k}) = \left(\st_n(\pi_1\dots \pi_n), \, \st_n(\pi_2\dots \pi_{n+1}), \, \dots, \, \st_n(\pi_{k+1} \dotsm \pi_{n+k}) \right).
    \] In other words, the value of $\lift^k \st_n$ on a permutation $\pi$ of length $n+k$ gives the value of $\st_n$ on each length-$n$ substring of $\pi$. When $k = 1$, we omit the superscript and just write $\lift \st_n$.
\end{defn}

Here are some basic properties of lifting:

\begin{proposition} \label{prop:lift-idempotent}
    We have $\lift^k (\lift^\ell \st_n) \equiv \lift^{k+\ell} \st_n$ for all $k, \ell \ge 0$.
\end{proposition}

\begin{proof}
    Let $\pi$ be a permutation of length $n+k+\ell$. Then, knowing $\lift^k (\lift^\ell \st_n)(\pi)$ is equivalent to knowing, for each substring $\pi'$ of length $n+\ell$, the value of $\st_n$ on each length-$n$ substring $\pi''$ of $\pi'$. But this is equivalent to knowing $\st_n(\pi'')$ for all length-$n$ substrings $\pi''$ of $\pi$, which is exactly $\lift^{k+\ell} \st_n$.
\end{proof}

\begin{proposition} \label{prop:refinement-lifting}
    Refinement respects lifting: if $\st_n \preceq \st'_n$, then $\lift^k \st_n \preceq \lift^k \st'_n$ for all $k \ge 0$.
\end{proposition}

\begin{proof}
    Let $\pi$ be a permutation of length $n+k$. Then $\lift^k \st_n(\pi)$ determines the value of $\st_n(\pi')$ for all length-$n$ substrings $\pi'$ of $\pi$. Since $\st_n \preceq \st'_n$, this determines $\st'_n(\pi')$ for all $\pi'$, so it determines $\lift^k \st'_n(\pi)$, as desired.
\end{proof}

We can now rephrase the definition of substring compatibility using refinement and lifting.

\begin{proposition} \label{prop:substring-lifting}
    Let $\st$ be a permutation statistic. Then the following are equivalent:

    \begin{enumerate}[(i)]
        \item $\st$ is substring-compatible.
        \item $\st_{n+k} \preceq \lift^k\st_n$ for all $n$ and $k$.
        \item $\st_{n+1} \preceq \lift \st_n$ for all $n$.
    \end{enumerate}
\end{proposition}

\begin{proof}
    Clearly, (ii) implies (iii), and the converse holds by induction on $k$ using \cref{prop:lift-idempotent,prop:refinement-lifting}. Thus, it suffices to show that (i) and (ii) are equivalent. Suppose $\st$ is substring-compatible, and let $\pi$ be a permutation of length $n+k$. Then $\st_{n+k}(\pi)$ determines $\st_n(\pi')$ for all length-$n$ substrings $\pi'$ of $\pi$, so $\st_{n+k}(\pi)$ determines $\lift^k \st_n(\pi)$, establishing (ii). Conversely, if (ii) holds, then $\st(\pi)$ determines $\st(\pi')$ for all substrings $\pi'$ of $\pi$ (by taking $n = |\pi'|$ and $k = |\pi| - |\pi'|$), so $\st$ is substring-compatible.
\end{proof}

\begin{remark} \label{remark:upwards-downwards}
    Suppose that $\st_{\le n-1}$ is a substring-compatible statistic, and we wish to extend this statistic to permutations of length $n$ such that $\st_{\le n}$ is also substring-compatible. Then \cref{prop:substring-lifting} implies that the set of possible values for $\st_n$ is downwards-closed with respect to refinement. (In particular, we can always take $\st_n$ to be discrete.) Furthermore, this set has a unique maximal element up to equivalence, namely $\lift \st_{n-1}$.
    
    By contrast, the opposite is true for shuffle compatibility: if $\st_{\le n-1}$ is (weakly) shuffle-compatible, then the possible values for $\st_n$ that make $\st_{\le n}$ (weakly) shuffle-compatible form an \emph{upwards}-closed set with respect to refinement. (In particular, we can always take $\st_n$ to be trivial.) However, this set does not have a unique minimal element in general.
\end{remark}

\subsection{The substring coalgebra}

We now define the substring coalgebra of a substring-compatible statistic.

\begin{defn}
    Let $\st$ be a substring-compatible permutation statistic. The \emph{substring coalgebra} of $\st$ is the graded $\Q$-coalgebra $\C_\st$ with basis $[\pi]_\st$ indexed by equivalence classes of $\sim_\st$, graded by length, and comultiplication given for $\pi = \pi_1 \dotsm \pi_n$ by \[
        \Delta([\pi]_\st) = \sum_{\sigma \tau = \pi} [\sigma]_\st \otimes [\tau]_\st = \sum_{i=0}^n [\pi_1 \dotsm \pi_i]_\st \otimes [\pi_{i+1} \dotsm \pi_n]_\st.
    \] (The counit of $\C_\st$ sends $[\pi]_\st$ to $1$ if $\pi$ is the empty permutation and $0$ otherwise.)
\end{defn}

Thus, $\Delta([\pi]_\st)$ gives the values of $\st$ over all ways of splitting $\pi$ into two substrings. It is well-defined by substring compatibility: if $[\pi]_\st = [\pi']_\st$, then $[\pi_i \dotsm \pi_j]_\st = [\pi'_i \dotsm \pi'_j]_\st$ for all $i$ and $j$, so the formula for $\Delta([\pi]_\st)$ does not depend on the choice of representative $\pi$. We also have an analog of \cref{prop:shuffle-algebra-refinement}, which is again easy to prove:

\begin{proposition} \label{prop:substring-coalgebra-refinement}
    If $\st$ and $\st'$ are substring-compatible statistics and $\st \preceq \st'$, then the map $\phi : \C_\st \to \C_{\st'}$ defined by $\phi([\pi]_\st) = [\pi]_{\st'}$ is a surjective coalgebra homomorphism. Thus, $\C_{\st'}$ is isomorphic to the quotient $\C_{\st} \, / \ker(\phi)$.
\end{proposition}

\subsection{Bicompatibility}

We now study how shuffle compatibility and substring compatibility interact. To this end, we make the following definition:

\begin{defn}
    A permutation statistic is \emph{(weakly) bicompatible} if it is (weakly) shuffle-compatible and also substring-compatible.
\end{defn}

\begin{example}
    The descent set, peak set, valley set, and trivial statistic are bicompatible; the discrete statistic and dual Knuth equivalence are weakly bicompatible.
\end{example}

If $\st$ is weakly bicompatible, then we can speak of both its shuffle algebra $\A_\st$ and its substring coalgebra $\C_\st$. These have the same underlying graded vector space, which is spanned by equivalence classes of $\sim_\st$ and graded by permutation length. In fact, the multiplication on $\A_\st$ and comultiplication on $\C_\st$ are compatible in the following sense:

\begin{proposition}
    Suppose that $\st$ is weakly bicompatible. Then the comultiplication of $\C_\st$ is a (graded) algebra homomorphism with respect to the multiplication of $\A_\st$.
\end{proposition}

\begin{proof}
    Given permutations $\pi$ and $\sigma$ with $\pi < \sigma$, we compute \begin{align*}
        \Delta\left( [\pi]_\st \cdot [\sigma]_\st \right)
        = \sum_{\tau \in \pi \sh \sigma} \Delta\left( [\tau]_\st \right)
        = \sum_{\tau \in \pi \sh \sigma} \sum_{\tau' \tau'' = \tau} [\tau']_\st \otimes [\tau'']_\st.
    \end{align*} For each choice of $\tau$, $\tau'$, $\tau''$, there is a unique way to write $\pi = \pi' \pi''$ and $\sigma = \sigma' \sigma''$ such that $\tau'$ is a shuffle of $\pi'$ and $\sigma'$, and $\tau''$ is a shuffle of $\pi''$ and $\sigma''$. Thus, we can reindex the sum as follows: \begin{align*}
        \Delta\left( [\pi]_\st \cdot [\sigma]_\st \right)
        &= \sum_{\pi' \pi'' = \pi} \sum_{\sigma' \sigma'' = \sigma} \sum_{\tau' \in \pi' \sh \sigma'} \sum_{\tau'' \in \pi'' \sh \sigma''} [\tau']_\st \otimes [\tau'']_\st
        \\ &= \sum_{\pi' \pi'' = \pi} \sum_{\sigma' \sigma'' = \sigma} \left( [\pi']_\st \cdot [\sigma']_\st \right) \otimes \left( [\pi'']_\st \cdot [\sigma'']_\st \right)
        \\ &= \left(\sum_{\pi' \pi'' = \pi} [\pi']_\st \otimes [\pi'']_\st \right) \cdot \left(\sum_{\sigma' \sigma'' = \sigma} [\sigma']_\st \otimes [\sigma'']_\st \right)
        \\ &= \Delta([\pi]_\st) \cdot \Delta([\sigma]_\st).
    \end{align*} Therefore $\Delta$ is indeed an algebra homomorphism.
\end{proof}

The preceding proposition shows that, when $\st$ is weakly bicompatible, we can combine the multiplication of $\A_\st$ and the comultiplication of $\C_\st$ to obtain a graded bialgebra. Because it is graded and connected,\footnote{Recall that a graded vector space $V = \bigoplus_{n \ge 0} V_n$ is \emph{connected} if $V_n$ is 1-dimensional.} this bialgebra always has a (unique) antipode map given by Takeuchi's formula \cite{takeuchi}. Hence, we obtain a graded Hopf algebra, which we denote $\H_\st$.

By combining \cref{prop:shuffle-algebra-refinement,prop:substring-coalgebra-refinement}, we obtain the following fact:

\begin{proposition} \label{prop:hopf-refinement}
    If $\st$ and $\st'$ are weakly bicompatible statistics and $\st \preceq \st'$, then the map $\phi : \H_\st \to \H_{\st'}$ defined by $\phi([\pi]_\st) = [\pi]_{\st'}$ is a surjective Hopf algebra homomorphism. Thus, $\H_{\st'}$ is isomorphic to the quotient $\H_{\st} \, / \ker(\phi)$.
\end{proposition}

\begin{example} \label{ex:Hdis-SSym}
    Recall from \cref{ex:ssym} that $\mathcal{A}_{\dis}$ is isomorphic as an algebra to the Malvenuto-Reutenauer Hopf algebra $\SSym$ \cite{malvenuto-reutenauer}. Now that we have upgraded $\A_{\dis}$ to a Hopf algebra $\H_{\dis}$, we can observe that this too is isomorphic to $\SSym$. Indeed, the comultiplication of $\SSym$ is given on the basis $\{\mathcal{F}_\pi\}$ by \[
        \Delta(\mathcal{F}_\pi) = \sum_{\sigma \tau = \pi} \mathcal{F}_{\stan(\sigma)} \otimes \mathcal{F}_{\stan(\tau)}.
    \] It follows that the map $[\pi]_{\dis} \mapsto \mathcal{F}_{\stan(\pi)}$, which we noted was an algebra isomorphism, is also a coalgebra map, so it is an isomorphism of Hopf algebras $\H_{\dis} \to \SSym$. Hence, for every weakly bicompatible statistic $\st$, the Hopf algebra $\H_\st$ is a quotient of $\SSym$. By their nature, such quotients are obtained by identifying basis elements of $\SSym$: that is, their corresponding biideals are generated by elements of the form $\mathcal{F}_{\pi} - \mathcal{F}_{\sigma}$.
\end{example}

\subsection{The descent set and peak set} \label{sec:des-pk-hopf}

To conclude this section, we return to the descent set and peak set, the examples which motivated our definition of substring compatibility. As described in the introduction, Stanley's work \cite{stanleythesis} (as reinterpreted by Gessel and Zhuang) shows that the shuffle algebra $\A_{\Des}$ is isomorphic to the ring of quasisymmetric functions $\QSym$. Specifically, there is an algebra isomorphism $\A_{\Des} \to \QSym$ which sends a basis element $[\pi]_{\Des} \in \A_{\Des}$ to a corresponding element\footnote{Here $\Comp(\pi)$ is a particular integer composition determined by the descent set (and length) of $\pi$. Specifically, if $\Des(\pi) = \{d_1 < d_2 < \dots < d_k\}$, then $\Comp(\pi) = (d_1, d_2-d_1, \dots, d_k - d_{k-1}, |\pi| - d_k)$.} $F_{\Comp(\pi)}$ of the \emph{fundamental basis} $\{F_\a\}$ of $\QSym$ \cite[Corollary 4.2]{gessel-zhuang}. Now that we have put a coalgebra structure on $\A_{\Des}$ to make it into $\H_{\Des}$, we can verify that this map also defines a Hopf algebra isomorphism $\H_{\Des} \to \QSym$ by using the well-known relationship between $\QSym$ and $\SSym$.

\begin{proposition} \label{prop:des-isomorphism}
    The linear map $\H_{\Des} \to \QSym$ defined by $[\pi]_{\Des} \mapsto F_{\Comp(\pi)}$ is an isomorphism of Hopf algebras.
\end{proposition}

\begin{proof}
    The linear map $\SSym \to \QSym$ defined by $\mathcal{F}_\pi \mapsto F_{\Comp(\pi)}$ is a surjective Hopf algebra map (see e.g.\ \cite[Corollary 8.1.14]{grinberg}). Composing this with the isomorphism $\H_{\dis} \to \SSym$ of \cref{ex:Hdis-SSym}, we obtain a surjective Hopf algebra map $\H_{\dis} \to \QSym$ given by $[\pi]_{\dis} \mapsto F_{\Comp(\pi)}$. But $\Comp(\pi)$ only depends on $\Des(\pi)$ (and $|\pi|$), so this map factors through the quotient $[\pi]_{\dis} \mapsto [\pi]_{\Des}$. Thus, $[\pi]_{\Des} \mapsto F_{\Comp(\pi)}$ is a Hopf algebra map. Since it maps basis elements to basis elements, it is an isomorphism.
\end{proof}

We now turn to the peak set. In \cite{stembridge}, Stembridge defined a family of ``peak'' functions $\{K_{n,\Lambda}\} \subset \QSym$, where $n$ is a positive integer and $\Lambda$ is a subset of $\{2, \dots, n-1\}$ that contains no two consecutive integers. (Thus, $\Lambda$ is the peak set of some length-$n$ permutation.) Then, he showed that the $K_{n, \Lambda}$ are linearly independent and that their span is closed under multiplication, so they form a basis for a subalgebra $\Pi$ of $\QSym$. Furthermore, Stembridge's work, as reinterpreted by Gessel and Zhuang, shows that the map $[\pi]_{\Pk} \mapsto K_{|\pi|, \Pk(\pi)}$ is an algebra isomorphism $\A_{\Pk} \to \Pi$.

Stembridge also showed that there is a surjective algebra map $\QSym \to \Pi$ that sends $F_{\Comp(\pi)}$ to $K_{|\pi|, \Pk(\pi)}$ for all permutations $\pi$. Later, Bergeron et al.\ \cite{bergeron} showed that $\Pi$ is also closed under the comultiplication of $\QSym$ (so it is a Hopf subalgebra of $\QSym$), and that this map $F_{\Comp(\pi)} \mapsto K_{|\pi|, \Pk(\pi)}$ is a Hopf algebra homomorphism. The proposition below follows from these results:

\begin{proposition}
    The linear map $\H_{\Pk} \to \Pi$ defined by $[\pi]_{\Pk} \mapsto K_{|\pi|, \Pk(\pi)}$ is an isomorphism of Hopf algebras.
\end{proposition}

\begin{proof}
    By composing Stembridge's map $F_{\Comp(\pi)} \mapsto K_{|\pi|, \Pk(\pi)}$ with the map $[\pi]_{\Des} \to F_{\Comp(\pi)}$ from \cref{prop:des-isomorphism}, we get a surjective Hopf algebra homomorphism $\H_{\Des} \to \Pi$ given by $[\pi]_{\Des} \mapsto K_{|\pi|, \Pk(\pi)}$. Since $K_{|\pi|, \Pk(\pi)}$ only depends on $\Pk(\pi)$ and $|\pi|$, this map factors through the quotient $[\pi]_{\Des} \mapsto [\pi]_{\Pk}$. The claim follows similarly to the previous proposition.
\end{proof}

\cref{fig:isomorphisms} summarizes the relationships between all the Hopf algebras considered here.

\begin{figure}
    \centering
    \begin{tikzcd}
    \mathcal{H}_{\text{dis}} \arrow[r, "\sim"] \arrow[d, two heads] & \text{SSym} \arrow[d, two heads] && {[\pi]_{\text{dis}}} \arrow[r, maps to] \arrow[d, maps to] & \mathcal{F}_{\pi} \arrow[d, maps to]    \\
    \mathcal{H}_{\text{Des}} \arrow[r, "\sim"] \arrow[d, two heads] & \text{QSym} \arrow[d, two heads] && {[\pi]_{\text{Des}}} \arrow[r, maps to] \arrow[d, maps to] & F_{\text{Comp}(\pi)} \arrow[d, maps to] \\
    \mathcal{H}_{\text{Pk}} \arrow[r, "\sim"] & \Pi && {[\pi]_{\text{Pk}}} \arrow[r, maps to] & {K_{|\pi|, \,\text{Pk}(\pi)}}          
    \end{tikzcd}
    \caption{The maps between $\SSym$, $\QSym$, $\Pi$, and the corresponding Hopf algebras of permutation statistics $\H_{\dis}$, $\H_{\Des}$, and $\H_{\Pk}$.}
    \label{fig:isomorphisms}
\end{figure}

\section{Characterizing bicompatible permutation statistics} \label{sec:characterizing-bicompatible}

In this final section, we consider the problem of characterizing all bicompatible permutation statistics. In fact, we conjecture that the only bicompatible statistics are the ones that we have already seen:

\begin{conjecture} \label{conj:bicompatible}
    Up to equivalence, the only bicompatible permutation statistics are the descent set $\Des$, the peak set $\Pk$, the valley set $\Val$, and the trivial statistic $\triv$.
\end{conjecture}

The following theorem, which we will spend most of this section proving, summarizes most of our progress towards proving \cref{conj:bicompatible}.

\begin{theorem} \label{thm:bicompatible-progress}
    Suppose that $\st$ is a bicompatible permutation statistic. Then either $\st \equiv \Des$, $\st \equiv \Pk$, $\st \equiv \Val$, or $\st_{\le 6} \equiv \triv_{\le 6}$ (that is, $\pi \sim_\st \sigma$ for all $\pi$ and $\sigma$ with $|\pi| = |\sigma| \le 6$).
\end{theorem}

Thus, the gap between \cref{conj:bicompatible} and \cref{thm:bicompatible-progress} lies in ruling out the existence of nontrivial bicompatible statistics which are trivial on length $\le\! 6$.

\subsection{A note on weakly bicompatible statistics} \label{subsec:note-weakly-bicompatible}

Before proving \cref{thm:bicompatible-progress}, let us take a moment to examine the analogous problem of characterizing all \emph{weakly} bicompatible permutation statistics. By contrast, these statistics appear to exist in great abundance: for example, one can construct a weakly bicompatible statistic from any substring-compatible statistic in the following way.

Given a permutation $\pi = \pi_1 \dotsm \pi_n$, let $\pi^{[a,b]}$ denote the (not necessarily contiguous) subsequence of $\pi$ consisting of all letters $\pi_i$ with $a \le \pi_i \le b$. Then an $[a,b]$-\emph{move} applied to $\pi$ consists of permuting the letters of $\pi^{[a,b]}$ within $\pi$. For example, if $\pi = 29546$, then $\pi^{[5,9]} = 956$, and so $\pi$ is related via a $[5,9]$-move to the permutation $\pi' = 26945$, where $\pi^{[5,9]}$ has been permuted to yield $(\pi')^{[5,9]} = 695$.

Next, fix a substring-compatible statistic $\st$. We say that an $[a,b]$-move that turns $\pi$ into $\pi'$ is \emph{$\st$-valid} if the corresponding subsequences $\pi^{[a,b]}$ and $(\pi')^{[a,b]}$ are $\st$-equivalent. Finally, we write $\pi \sim \sigma$ if $\sigma$ can be reached from $\pi$, or vice versa, by a sequence of $\st$-valid $[a,b]$-moves (where $a$ and $b$ can vary over different moves). It is clear that $\sim$ is an equivalence relation.

\begin{theorem} \label{thm:weakly-bicompatible-construction}
    The equivalence relation $\sim$ defines a weakly bicompatible statistic.
\end{theorem}

\begin{proof}
    Let $\pi$, $\sigma$, and $\sigma'$ be permutations with $\pi < \sigma$ and $\pi < \sigma'$, and suppose that $\sigma$ and $\sigma'$ are related by an $\st$-valid $[a,b]$-move. Then in any shuffle $\tau$ of $\pi$ and $\sigma$, we have $\tau^{[a,b]} = \sigma^{[a,b]}$ because none of the letters of $\pi$ lie in the range $[a,b]$. Thus, the same $[a,b]$-move can be applied to $\tau$, yielding a corresponding shuffle of $\pi$ and $\sigma'$. This proves that $\pi \sh \sigma \sim \pi \sh \sigma'$. (Similarly, if we have permutations $\pi$, $\pi'$, and $\sigma$ with $\pi < \sigma$ and $\pi' < \sigma$ where $\pi$ and $\pi'$ are related by an $\st$-valid $[a,b]$-move, then $\pi \sh \sigma \sim \pi' \sh \sigma$.) Then the statistic defined by $\sim$ is weakly shuffle-compatible by transitivity.
    
    Next, suppose that $\pi$ and $\sigma$ are related by an $\st$-valid $[a,b]$-move, and let $\pi'$ and $\sigma'$ be corresponding substrings of $\pi$ and $\sigma$, respectively. Then it is not hard to see that $(\pi')^{[a,b]}$ and $(\sigma')^{[a,b]}$ are corresponding substrings of $\pi^{[a,b]}$ and $\sigma^{[a,b]}$, respectively. Since the original $[a,b]$-move was $\st$-valid, we have $\pi^{[a,b]} \sim_{\st} \sigma^{[a,b]}$, and then substring compatibility of $\st$ implies that $(\pi')^{[a,b]} \sim_\st (\sigma')^{[a,b]}$. Thus, $\pi'$ and $\sigma'$ are also related by an $\st$-valid $[a,b]$-move, so the statistic defined by $\sim$ is substring-compatible (again by transitivity).
\end{proof}

As an example, suppose $\st$ is the statistic with only two nontrivial equivalences, namely $132 \sim_{\st} 231$ and $213 \sim_{\st} 312$. (Thus, $\st_n$ is discrete for all $n \neq 3$.) Then it is not hard to see that $\st$ is substring-compatible, and furthermore, $\st$-valid $[a,b]$-moves are exactly the same as dual Knuth moves. Therefore, \cref{thm:weakly-bicompatible-construction} generalizes the fact that dual Knuth equivalence is weakly bicompatible (as we saw in \cref{ex:dual-knuth-weakly-shuffle,ex:dual-knuth-substring}).

\subsection{The descent set}

We now embark on the proof of \cref{thm:bicompatible-progress}. The main idea is to work ``one level at a time'': given a bicompatible statistic $\st_{\le n-1}$ for some $n$, we determine all possible ways to extend it to length-$n$ permutations by defining $\st_n$ such that $\st_{\le n}$ is bicompatible. Then bicompatible statistics (defined on all permutations) are obtained by iteratively performing this process for all $n$. However, as we will see, some bicompatible statistics $\st_{\le n-1}$ admit no bicompatible extensions to length-$n$ permutations. (See \cref{remark:upwards-downwards} for another perspective on this process.)

We begin with permutations of length $2$. For any permutation statistic $\st$, either $12 \sim_\st 21$ or $12 \not\sim_\st 21$. Our first order of business is to show that if $\st$ is bicompatible and $12 \not \sim_\st 21$, then $\st \equiv \Des$. To do so, we will need a couple of lemmas.

\begin{lemma}\label{lemma:Des-distinct}
    Let $\pi$ and $\sigma$ be disjoint permutations with $|\sigma| = 1$. Then all shuffles of $\pi$ and $\sigma$ have distinct descent sets.\footnote{In fact, a stronger statement is true: all shuffles of $\pi$ and $\sigma$ have distinct major indices. This can be shown by casework or by using the identity \cite[Equation (1.2)]{gessel-zhuang}. However, the given statement has an easy proof and will suffice for our purposes.}
\end{lemma}

\begin{proof}
    Let $\pi = \pi_1 \dotsm \pi_n$ and $\sigma = \sigma_1$. Suppose that two distinct shuffles of $\pi$ and $\sigma$, say \[
        \pi_1 \dotsm \pi_i \sigma_1 \pi_{i+1} \dotsm \pi_n \quad \text{and} \quad \pi_1 \dotsm \pi_j \sigma_1 \pi_{j+1} \dotsm \pi_n
    \] for $0 \le i < j \le n$, have the same descent set. Then the corresponding substrings \[
        \sigma_1 \pi_{i+1} \dotsm \pi_j \quad \text{and} \quad \pi_{i+1} \dotsm \pi_j \sigma_1
    \] must also have the same descent set. Assuming without loss of generality that $\sigma_1 < \pi_{i+1}$, comparing descent sets shows that $\sigma_1 < \pi_{i+1} < \ldots < \pi_j < \sigma_1$, which is impossible.
\end{proof}

\begin{lemma}\label{lemma:Des-maximal}
    Among shuffle-compatible statistics, $\Des$ is maximally refined: that is, there is no shuffle-compatible statistic $\st \not\equiv \Des$ with $\st \preceq \Des$.
\end{lemma}

\begin{proof}
    Fix a shuffle-compatible statistic $\st$ with $\st \preceq \Des$. We will show that $\Des_n \preceq \st_n$ by induction on $n$. The statement is obvious for $n \le 2$. Let $\pi = \pi_1 \dotsm \pi_n$ and $\sigma = \sigma_1 \dotsm \sigma_n$ be permutations with $\Des(\pi) = \Des(\sigma)$ for $n \ge 3$. Set $\pi' = \pi_1\dotsm \pi_{n-1}$ and $\sigma' = \sigma_1\dotsm \sigma_{n-1}$. Then $\Des(\pi') = \Des(\sigma')$, so by the inductive hypothesis, $\st(\pi') = \st(\sigma')$. This implies that \[
        \pi' \sh \pi_n \sim_\st \sigma' \sh \sigma_n
    \] by shuffle-compatibility of $\st$ (where $\pi_n$ and $\sigma_n$ are considered to be length-$1$ permutations). In particular, since $\pi \in \pi' \sh \pi_n$, there must exist some $\sigma^* \in \sigma' \sh \sigma_n$ such that $\st(\pi) = \st(\sigma^*)$. Then $\Des(\sigma) = \Des(\pi) = \Des(\sigma^*)$, since $\st$ refines $\Des$. But by \cref{lemma:Des-distinct}, all the elements of $\sigma' \sh \sigma_n$ have distinct descent sets, so we must have $\sigma^* = \sigma$. Hence, $\st(\pi) = \st(\sigma)$.
\end{proof}

\begin{theorem} \label{thm:Des-bicompatible-case}
    If $\st$ is bicompatible and $12 \not\sim_\st 21$, then $\st \equiv \Des$.
\end{theorem}
\begin{proof}
    Since $12 \not\sim_\st 21$, we have $\st_2 \equiv \Des_2$. Then, since $\st$ is substring-compatible, we have \[\st_n \preceq \lift^{n-2} \st_2 \equiv \lift^{n-2} \Des_2 \equiv \Des_n\] for all $n$, so $\st \preceq \Des$. (In other words, $\st(\pi)$ determines $\st(\pi_i\pi_{i+1})$ for all $1 \le i \le n-1$. Since $\st(12) \neq \st(21)$, this determines whether $\pi_i < \pi_{i+1}$ or $\pi_i > \pi_{i+1}$, so it determines the descent set of $\pi$.) Then \cref{lemma:Des-maximal} forces $\st \equiv \Des$.
\end{proof}

\subsection{The peak set and valley set}

We now assume that $12 \sim_\st 21$; that is, $\st_{\le 2} \equiv \triv_{\le 2}$. Next, we examine the possibilities for $\st_3$.

\begin{lemma} \label{lemma:bicompatible-3}
    If $\st$ is bicompatible and $12 \sim_\st 21$, then either $\st_{\le 3} \equiv \Pk_{\le 3}$, $\st_{\le 3} \equiv \Val_{\le 3}$, or $\st_{\le 3} \equiv \triv_{\le 3}$.
\end{lemma}

\begin{proof}
    By shuffle compatibility, we have $12 \sh 3 \sim_{\st_3} 13 \sh 2$. The corresponding shuffle sets are $\{123, 132, 312\}$ and $\{123, 132, 213\}$, respectively, which have $123$ and $132$ in common, so we can cancel them and conclude that $312 \sim_{\st_3} 213$. Similarly, we have $13 \sh 2 \sim_{\st_3} 23 \sh 1$, which gives $132 \sim_{\st_3} 231$, and $23 \sh 1 \sim_{\st_3} 21 \sh 3$ (using the fact that $\st_{\le 2}$ is trivial), which gives $123 \sim_{\st_3} 321$. Thus, $\st_3$ must be refined by the statistic whose equivalence classes are \[
        A = \{213, 312\}, \quad B = \{123, 321\}, \quad C = \{132, 231\}.
    \]
    Next, we claim that $B$ cannot be an equivalence class of $\st_3$ on its own. To prove this claim, observe that $12 \sh 34 \sim_{\st_4} 21 \sh 34$, again by shuffle-compatibility. Since $\st$ is substring-compatible, we have $\st_4 \preceq \lift \st_3$, so we get \[12 \sh 34 \sim_{\lift \st_3} 21 \sh 34.\] In particular, $1234$ is a shuffle of $12$ and $34$, so there must exist $\sigma = \sigma_1 \sigma_2 \sigma_3 \sigma_4 \in 21 \sh 34$ such that $\lift \st_3(1234) = \lift \st_3(\sigma)$. Since \[
        \lift \st_3(1234) = \left(\st_3(123), \, \st_3(234)\right) = \left(\st_3(123), \, \st_3(123) \right),
    \] we must have $\sigma_1\sigma_2\sigma_3 \sim_{\st_3} 123$ and $\sigma_2\sigma_3\sigma_4 \sim_{\st_3} 123$. Note that $123$ lies in the set $B$. However, the only length-$4$ permutations whose length-$3$ substrings both lie in $B$ (after standardization) are $1234$ and $4321$, which are not shuffles of $21$ and $34$, so the equivalence class of $123$ must be strictly larger than $B$. Therefore, the equivalence classes of $\st_3$ are either $A \cup B$ and $C$ (which gives $\Pk_3$), $A$ and $B \cup C$ (which gives $\Val_3$), or just $A \cup B \cup C$ (which gives $\triv_3$).
\end{proof}

Notice that in order to prove the preceding lemma, we needed to use the fact that $\st_{\le 3}$ must extend to a bicompatible statistic on length $\le\!4$. (Indeed, taking the equivalence classes of $\st_3$ to be $A$, $B$, and $C$ would have yielded a bicompatible statistic $\st_{\le 3}$, but one that cannot be extended to length $4$.)

We now have three cases for the equivalence relation induced by $\st_3$. We will now prove the following result, completely resolving the first case:

\begin{theorem} \label{thm:Pk-bicompatible-case}
    If $\st$ is bicompatible and $\st_{\le 3} \equiv \Pk_{\le 3}$, then $\st \equiv \Pk$.
\end{theorem}

Before proving this theorem, we note that it implies the analogous result for the second case, due to a symmetry relating the peak set and valley set.

\begin{corollary}
    If $\st$ is bicompatible and $\st_{\le 3} \equiv \Val_{\le 3}$, then $\st \equiv \Val$.
\end{corollary}

\begin{proof}
    Define $\st'(\pi) = \st(\pi^c)$, where $\pi^c$ is the \emph{complement} of $\pi$, that is, the permutation obtained by simultaneously replacing the $i^{\text{th}}$ smallest letter of $\pi$ with the $i^{\text{th}}$ largest letter of $\pi$ for all $i$. (For instance, when $\pi$ is standard, we have $(\pi^c)_i = |\pi|+1 - \pi_i$.) Note that $(\pi^c)^c = \pi$ and $\Val(\pi) = \Pk(\pi^c)$ for all $\pi$. Then, it is not hard to check that $\st'$ is also bicompatible and $\st'_{\le 3} \equiv \Pk_{\le 3}$. Hence $\st' \equiv \Pk$ by \cref{thm:Pk-bicompatible-case}, and therefore $\st \equiv \Val$.
\end{proof}

We now move to the proof of \cref{thm:Pk-bicompatible-case}. For an integer $n \ge 4$, define a statistic $\sPk_n$ as follows: for standard permutations $\pi$, \[
    \sPk_n(\pi) = \begin{cases}
        \pi^{-1}(1) \text{ mod } 2 & \text{if } \Pk(\pi) = \varnothing, \\
        \Pk(\pi) & \text{otherwise}.
    \end{cases}
\] (By $\pi^{-1}(1)$ we mean the index $i$ such that $\pi_i = 1$.) Thus, $\sPk_n$ is obtained from $\Pk_n$ by dividing the equivalence class $\{ \pi : \Pk(\pi) = \varnothing\}$ into two subclasses. It follows that $\sPk_n$ refines $\Pk_n$, and moreover, this is a covering relation.

\begin{lemma} \label{lemma:Pk-generates-sPk}
    Fix $n \ge 4$. If $\st_{\le n}$ is bicompatible with $\st_{\le n-1} \equiv \Pk_{\le n-1}$, then $\sPk_n \preceq \st_n$.
\end{lemma}

\begin{proof}
    Since $\st_{\le n}$ is bicompatible, we have \[
        \st_n \preceq \lift^{n-3} \st_3 \equiv \lift^{n-3} \Pk_3 \equiv \Pk_n\!.
    \] Thus, if $S \sim_{\st_n} S'$ for some sets $S$ and $S'$, then each permutation in $S$ must be $\st$-equivalent to a permutation in $S'$ \emph{with the same peak set}. We will repeatedly use this fact in the proof.

    Consider two $\sPk_n$-equivalent permutations. Since $\sPk_n \preceq \Pk_n$, these permutations must have the same peak set, say $\Lambda \subseteq \{2, 3, \dots, n-1\}$. We take cases on $\Lambda$.

    \bigskip

    \textbf{Case 1}: $\Lambda = \varnothing$. Notice that every standard length-$n$ permutation with peak set $\varnothing$ is uniquely specified by a subset $L \subseteq \{2, \dots, n\}$, which consists of all letters of $\pi$ that appear before $1$. (These letters must appear in decreasing order, and the remaining letters must appear after $1$ in increasing order.) Denote this permutation by $\pi_L$; then \[
        \sPk(\pi_L) = \pi_L^{-1}(1) = \left(|L| + 1\right) \text{ mod } 2.
    \] Thus, we want to show that if $|L| \equiv |L'| \pmod 2$, then $\pi_L \sim_{\st_n} \pi_{L'}$.

    Fix any $L \subseteq \{2, \dots, n\}$ with $|L| \ge 2$, and choose distinct elements $i, j \in L$. Let $\pi_L - i$ be the permutation of length $n-1$ obtained by deleting the letter $i$ from $\pi_L$; analogously define $\pi_{L \setminus \{i\}} - j$. These permutations both have peak set $\varnothing$, so by shuffle-compatibility, we have \begin{align}
        i \sh (\pi_L - i) \,\sim_{\st_n}\, j \sh (\pi_{L \setminus \{i\}} - j). \label{eq:empty-peak-set-shuffle}
    \end{align} Among shuffles of $i$ and $\pi_L - i$, the only ones with empty peak set are $\pi_L$ and $\pi_{L \setminus \{i\}}.$ Similarly, among shuffles of $j$ and $\pi_{L \setminus \{i\}} - j$, the only ones with empty peak set are $\pi_{L \setminus \{i\}}$ and $\pi_{L \setminus \{i,j\}}$. Thus, \eqref{eq:empty-peak-set-shuffle}, together with the observation from the beginning of the proof, forces \[
        \pi_L \sim_{\st_n} \pi_{L \setminus \{i,j\}}.
    \] This proves the desired fact by transitivity, since for sets $L$ and $L'$ with $|L| \equiv |L'| \pmod 2$, each set can be obtained from the other by repeatedly adding or removing pairs of elements.

    \bigskip
    
    \textbf{Case 2}: $2 \in \Lambda$. Suppose $\pi = \pi_1 \dotsm \pi_n$ and $\sigma = \sigma_1 \dotsm \sigma_n$ both have peak set $\Lambda$. Then $\pi_2 \dotsm \pi_n$ and $\sigma_2 \dotsm \sigma_n$ have the same peak set, so they are $\st$-equivalent. Therefore, \[
        \pi_1 \sh \pi_2 \dotsm \pi_n \, \sim_{\st_n} \, \sigma_1 \sh \sigma_2 \dotsm \sigma_n.
    \] Therefore, $\pi$ (which is a shuffle of $\pi_1$ and $\pi_2 \dotsm \pi_n$) must be $\st_n$-equivalent to some shuffle of $\sigma_1$ and $\sigma_2 \dotsm \sigma_n$ with peak set $\Lambda$. Since $2 \in \Lambda$, we have $\sigma_1 < \sigma_2 > \sigma_3$. Thus, among shuffles of $\sigma_1$ and $\sigma_2 \dotsm \sigma_n$, the only permutation with a peak at $2$ is $\sigma$ itself: indeed, every other shuffle begins with either $\sigma_2\sigma_1$ or $\sigma_2\sigma_3$ and hence has a descent at position $1$. Thus $\pi \sim_{\st_n} \sigma$.

    \bigskip
    
    \textbf{Case 3}: $\Lambda \neq \varnothing$ and $2 \notin \Lambda$.

    Let $k = \min(\Lambda) \ge 3$. Fix any permutation $\pi = \pi_1 \dotsm \pi_n$ with peak set $\Lambda$ such that \begin{align}
        \pi_1 < \dots < \pi_{k-1} < \pi_{k+1} < \pi_k < \pi_{k+2}. \label{eq:peak-case-inequality-string}
    \end{align} (It is not hard to see that such a $\pi$ exists. If $k = n-1$, ignore the last inequality with $\pi_{k+2}$.) We will show that every length-$n$ permutation with peak set $\Lambda$ is $\st$-equivalent to $\pi$.

    Let $\pi' = \pi_2 \pi_1 \pi_3 \dotsm \pi_n$, the result of switching the first two letters in $\pi$. Note that $\pi'$ also has peak set $\Lambda$. We first show that $\pi'$ is $\st$-equivalent to $\pi$. To do so, we write \[
        \pi_k \sh (\pi - \pi_k) \,\, \sim_{\st_n} \,\, \pi_{k+1} \sh (\pi' - \pi_{k+1}).
    \] (The permutations $\pi - \pi_k$ and $\pi' - \pi_{k+1}$ have the same peak set due to \eqref{eq:peak-case-inequality-string} and are thus $\st$-equivalent.) Thus, $\pi'$ must be $\st$-equivalent to some shuffle of $\pi_k$ and $\pi - \pi_k$ whose peak set is $\Lambda$. Among shuffles of $\pi_k$ and $\pi - \pi_k$, the only permutation with a peak at $k$ is $\pi$ itself: every other shuffle has either $\pi_{k-1} < \pi_{k+1}$, $\pi_{k+1} < \pi_k$, or $\pi_{k+1} < \pi_{k+2}$ in positions $k$ and $k+1$. Thus, $\pi \sim_{\st_n} \pi'$, as desired.

    Now, let $\sigma = \sigma_1 \dotsm \sigma_n$ be any permutation with peak set $\Lambda$. Then \[
        \pi_1 \sh \pi_2 \dotsm \pi_n \,\, \sim_{\st_n} \,\, \sigma_1 \sh \sigma_2 \dotsm \sigma_n.
    \] Thus, $\sigma$ must be $\st$-equivalent to some shuffle of $\pi_1$ and $\pi_2 \dotsm \pi_n$ with peak set $\Lambda$. We claim that the only such shuffles are $\pi$ and $\pi'$. Indeed, every other shuffle of $\pi_1$ and $\pi_2 \dotsm \pi_n$ either starts with $\pi_2 \dotsm \pi_i \pi_1$ for some $3 \le i \le k$ or starts with $\pi_2 \dotsm \pi_{k+1}$. In the former case, it has a peak at position $i-1$ (since $\pi_{i-1} < \pi_i > \pi_1$), which is strictly less than $k$, so its peak set is not $\Lambda$. In the latter case, it has a descent $w_k w_{k+1}$ at position $k-1$, so its peak set cannot be $\Lambda$ either. Therefore, either $\pi \sim_{\st_n} \sigma$ or $\pi' \sim_{\st_n} \sigma$. But we know that $\pi \sim_{\st_n} \pi'$, so either way we get $\pi \sim_{\st_n} \sigma$. This completes this case and the proof.
\end{proof}

\begin{proof}[Proof of \cref{thm:Pk-bicompatible-case}]
    We show that $\st_n \equiv \Pk_n$ by induction on $n \ge 3$; the base case is given. Fix $n \ge 4$. If $\st$ is bicompatible and $\st_{\le n-1} \equiv \Pk_{\le n-1}$, then either $\st_n \equiv \sPk_n$ or $\st_n \equiv \Pk_n$ by \cref{lemma:Pk-generates-sPk}. Therefore, it remains only to show that the choice $\st_n \equiv \sPk_n$ does not extend to a bicompatible statistic on all permutations.
    
    Suppose that $\st_n \equiv \sPk_n$. Note that the length-$n$ permutations $1 \dotsm n$ and $3214 \dotsm n$ are $\sPk_n$-equivalent. Then by bicompatibility, we have \[
        1 \dotsm n \, \sh \, (n+1) \,\, \sim_{\lift \st_n} \,\, 3214 \dotsm n \, \sh \, (n+1).
    \] The permutation $1 \dotsm (n+1)$ appears on the left-hand side, and \[\lift \st_n(1 \dotsm (n+1)) = (\st_n(1 \dotsm n), \, \st_n(1 \dotsm n)) = (1 \text{ mod } 2, \, 1 \text{ mod } 2).\] It is not hard to show that the only standard $\pi$ with $\lift\st_n(\pi) = (1 \text{ mod } 2, \, 1 \text{ mod } 2)$ are $1 \dotsm (n+1)$ and, for $n$ odd, its reverse $(n+1) \dotsm 1$. However, neither of these appear on the right-hand side, a contradiction, so we must have $\st_n \equiv \Pk_n$.
\end{proof}


\subsection{Beyond}

We are left with the case in which $\st$ is bicompatible and trivial on permutations of length at most $3$. Unfortunately, we have not been able to show that $\st$ must be trivial given these conditions, nor (of course) do we know of any counterexample.

A reasonable approach to finishing the proof of \cref{conj:bicompatible} might involve showing for each $n \ge 4$ that if $\st$ is bicompatible and $\st_{\le n-1} \equiv \triv_{\le n-1}$, then $\st_n \equiv \triv_n$. The difficulty with this approach is that there are many bicompatible statistics $\st_{\le n}$, defined only for permutations of length at most $n$, such that $\st_{\le n-1}$ is trivial but $\st_n$ is not. For example, O\u{g}uz \cite{oguz} constructed a shuffle-compatible statistic $\Psi$ with $\Psi_{\le 3} \equiv \triv_{\le 3}$ but $\Psi_4 \not\equiv \triv_4$, so $\Psi_{\le 4}$ is bicompatible. However, as we will see, $\Psi_{\le 4}$ cannot be extended to a bicompatible statistic on all permutations: that is, there is no bicompatible $\st$ with $\st_{\le 4} \equiv \Psi_{\le 4}$.

Despite this, by combining additional results with computer calculations, we have been able to show that if $\st$ is bicompatible and $\st_{\le 3}$ is trivial, then in fact $\st_{\le 6}$ must be trivial. We now describe our methods.

\subsubsection{A number-theoretic restriction}

First, we present a restriction which applies even if $\st_{\le n}$ does not extend to a bicompatible statistic on all permutations (in fact, even if it is only \emph{weakly} bicompatible).

\begin{theorem} \label{thm:number-theoretic}
    Suppose for $n \ge 4$ that $\st_{\le n}$ is weakly bicompatible and $\st_{\le n-1}$ is trivial.
    
    \begin{enumerate}[(i)]
        \item If $n = p^k$ for some prime $p$ and some $k \ge 1$, then $\st_n$ has at most $p$ equivalence classes.
        \item Otherwise, $\st_n$ must be trivial.
    \end{enumerate}
\end{theorem}

\begin{proof}
    Let $C$ be an equivalence class of $\st_n$, including only standard permutations (so that $C \subseteq S_n$). Let $a \in [n-1]$, and consider all shuffle sets $\pi \sh \sigma$ where $\pi$ has letter set $[a]$ and $\sigma$ has letter set $\{a+1, \dots, n\}$. Because $\st_{\le n}$ is weakly shuffle-compatible and $\st_{\le n-1}$ is trivial, all these shuffle sets are $\st_n$-equivalent, so there are the same number of permutations in each shuffle set belonging to $C$. Furthermore, these shuffle sets form a partition of $S_n$. Thus, the cardinality of $C$ must be divisible by the number of shuffle sets, which is $a!(n-a)!$.

    Since this holds for all $a$, it follows that the cardinality of $C$ is divisible by \[
        \operatorname{lcm}\left(1!(n-1)!,\, 2!(n-2)!, \,\dots,\, (n-1)!1!\right) = \frac{n!}{\gcd\left(\binom{n}{1}, \binom{n}{2}, \dots, \binom{n}{n-1}\right)}.
    \] Furthermore, since $C$ was arbitrary, the number of equivalence classes of $\st_n$ is at most $\gcd\left(\binom{n}{1}, \binom{n}{2}, \dots, \binom{n}{n-1}\right)$. The result then follows from the elementary fact that \[\gcd\left(\binom{n}{1}, \binom{n}{2}, \dots, \binom{n}{n-1}\right) = \begin{cases}
        p & \text{if } n \text{ is a power of a prime } p, \\
        1 & \text{otherwise}.
    \end{cases}
    \]
\end{proof}

This theorem essentially reduces the problem to the case in which $n$ is a prime power (although it is not clear whether this is much easier than the general case). In particular, the calculations described below will show that if $\st$ is bicompatible and $\st_{\le 3}$ is trivial, then $\st_4$ and $\st_5$ must also be trivial; then \cref{thm:number-theoretic} implies that $\st_6$ must be trivial.

\subsubsection{Computer calculations}

We employed the SAT solver \texttt{clasp} \cite{clasp} to compute, for $n = 4$ and $n = 5$, all shuffle-compatible permutation statistics $\st_{\le n}$ such that $\st_{\le n-1}$ is trivial. Specifically, we used a wrapper for \texttt{clasp} in Python called \texttt{claspy} \cite{claspy}, which allows the user to define variables taking values in a set of arbitrary size; these variables are internally encoded as an array of booleans, exactly one of which must take the value \texttt{true}.

Here are some details of the implementation. For each standard length-$n$ permutation $\pi$, we initialize a variable $v_{\pi}$ (representing the equivalence class $[\pi]_\st$) which takes its value in the set $[n]$. (Note that by \cref{thm:number-theoretic}, $\st_n$ has at most $n$ equivalence classes.) Next, for all positive integers $a \le b$ with $a + b = n$, we enumerate all shuffle sets $\sigma \sh \tau$ with $|\sigma| = a$ and $|\tau| = b$ in which the letters of $\sigma$ and $\tau$ comprise the set $[n]$. For each shuffle set, we define auxiliary variables $w_{\sigma, \tau, i}$ for $i \in [n]$ by \[
    w_{\sigma, \tau, i} = \left| \{ \pi \in \sigma \sh \tau : v_\pi = i\} \right|.
\] Finally, shuffle compatibility of $\st_{\le n}$ is equivalent to the condition that all the variables $w_{\sigma, \tau, i}$ are equal for each fixed triple $(a, b, i)$. After encoding this condition as a boolean formula, we asked \texttt{clasp} to find all satisfying assignments, thus computing all possible values for $\st_n$.

As noted, we were able to run this program only for $n = 4$ and $n = 5$. (As with many computations involving permutations, the complexity grows super-exponentially with $n$.) For $n = 4$, the program computed $11$ different statistics $\st_4$ such that $\st_{\le 4}$ is shuffle-compatible (one of which is O\u{g}uz's example from above), and for $n = 5$, it computed $10928$ different statistics $\st_5$ such that $\st_{\le 5}$ is shuffle-compatible. However, in both cases, we were able to check that no choice for $\st_n$ besides the trivial statistic can be extended to a bicompatible statistic on all permutations. To do this, we used the following (easy) observation:

\begin{lemma}
    Suppose that $\st$ is bicompatible. If $\pi$, $\pi'$ are $\st$-equivalent permutations of length $2$ and $\sigma$, $\sigma'$ are $\st$-equivalent permutations of length $n$, then \[
        \pi \sh \sigma \sim_{\lift^2 \st_n} \pi' \sh \sigma'.
    \]
\end{lemma}

\begin{proof}
    Since $\st$ is shuffle-compatible, we have $\pi \sh \sigma \sim_{\st_{n+2}} \pi' \sh \sigma'$; then the claim follows by substring compatibility.
\end{proof}

Notice that this condition only depends on the values of $\st_{\le n}$, so it can be checked for each statistic $\st_n$ generated by our program. Indeed, for $n=4$ and $n=5$, we found that the above condition was always violated at least once if $\st_n$ was nontrivial; thus, $\st_4$ must be trivial, and then $\st_5$ must also be trivial. Finally, since $6$ is not a power of a prime, \cref{thm:number-theoretic} shows that $\st_6$ must be trivial as well. This proves \cref{thm:bicompatible-progress}.

\section{Acknowledgments}

The author thanks Ricky Liu for many helpful comments on a draft of this paper, and in particular, for suggesting the generalization of dual Knuth equivalence found in \cref{subsec:note-weakly-bicompatible}.

\printbibliography

\end{document}